\documentclass[12pt]{article}
\usepackage[top=3.2cm,bottom=3.2cm,left=2.5cm,right=2.5cm]{geometry}
\usepackage{amssymb}
\usepackage{amsmath,amsthm}
\usepackage[latin1]{inputenc}
\usepackage[dvips]{graphicx}
\usepackage{hyperref}
\usepackage{color}
\usepackage{mathrsfs}
\usepackage{enumerate}
\usepackage{enumitem}
\usepackage{tikz}
\usepackage{xifthen}
\usepackage{verbatim}

\hypersetup{colorlinks=true, linkcolor=blue, citecolor=blue,urlcolor=blue}

\newtheorem{remark}{Remark}[section]

\newtheorem{lemma}[remark]{Lemma}
\newtheorem{theorem}[remark]{Theorem}

\title{Constructive characterizations concerning total outer-independent domination in subdivision trees}
\author{A. Cabrera-Mart\'inez, J.L. L\'opez-Carmona, A. Serrano-D\'iaz\\
	{\small Universidad de C\'ordoba, Departamento de Matem\'aticas} \\
	{\small Campus de Rabanales, 14071, C\'ordoba, Spain.} \\{\small acmartinez@uco.es, 2locaj@uco.es, aserrano1@uco.es}
}

\date{ }
\begin{document}
\maketitle

\begin{abstract}
Let $G$ be a nontrivial connected graph with vertex set $V(G)$. A set of vertices $D\subseteq V(G)$ is called a total outer-independent dominating set of $G$ if every vertex of $G$ is adjacent to at least one vertex in $D$, and  $V(G)\setminus D$ is an independent set of $G$. The total outer-independent domination number of $G$, denoted by $\gamma_t^{oi}(G)$, is the minimum cardinality among all total outer-independent dominating sets of $G$. The subdivision graph of $G$, denoted by $\mathtt{S}(G)$, is the graph obtained from $G$ by subdividing every edge exactly once.
Cabrera-Mart\'inez et al. [On the total outer-independent domination number of subdivision graphs, Comput. Appl. Math. 45 (2026) 315] proved that $\tfrac{4n(T)-l(T)-s(T)}{3}\leq \gamma_{t}^{oi}(\mathtt{S}(T))\leq \tfrac{4n(T)-l(T)+s(T)-2}{3}$ for any nontrivial tree $T$ of order $n(T)$ with $l(T)$ leaves and $s(T)$ support vertices. In this paper, we provide constructive characterizations of the families of trees that attain these bounds.
\end{abstract}


\noindent
{\it Math. Subj. Class. (2020)}: 05C05, 05C69, 05C75.

\section{Introduction}

Let $G(V(G),E(G))$ be a simple connected graph of order $n(G)=|V(G)|$. We denote its vertex set by $V(G)=\cup_{i\in \{1,\ldots,n(G)\}}\{v_i\}$, and let us define $V_E(G)=\{v^{i,j} : v_iv_j\in E(G)\}$ (remark that $v^{i,j}=v^{j,i}$). Given any vertex $v_i\in V(G)$, $N_G(v_i)$ represents the open neighborhood of vertex $v_i$, that is, $N_G(v_i)=\{v_j\in V(G): v_iv_j\in E(G)\}$. A vertex $v_i\in V(G)$ is called a \emph{leaf} of $G$ if $|N_G(v_i)|=1$, and $v_i$ is a \emph{support vertex} of $G$ if it is adjacent to a leaf. The sets of leaves and support vertices are denoted by $\mathcal{L}(G)$ and $\mathcal{S}(G)$, respectively. Moreover, we write $l(G)=|\mathcal{L}(G)|$ and $s(G)=|\mathcal{S}(G)|$. By convention, for the path $P_2$ we assume that one of its two vertices is a support vertex and the other is a leaf.

\vspace{.2cm}

\noindent
A \emph{total outer-independent dominating set} (TOIDS) of a nontrivial connected graph $G$ is a set $D\subseteq V(G)$ which satisfies that  every vertex of $G$ is adjacent to at least one vertex in $D$ and that  $V(G)\setminus D$ is an independent set of $G$. The \emph{total outer-independent domination number} of $G$ is defined as $\gamma_t^{oi}(G)=\min\{|D|: D \text{ is a TOIDS of } G\}.$
A $\gamma_{t}^{oi}(G)$-\emph{set} is a TOIDS of $G$ with cardinality $\gamma_{t}^{oi}(G)$. 
The concept of total outer-independent domination in graphs was introduced in 2012 by Soner et al.\,\cite{TCOI-Soner}, and has subsequently been investigated in several works, including  \cite{Cabrera-TCID-regular,Cabrera-TCID-product,Cabrera2022,Cabrera2019,Cabrera-TOID-2026,TOID-lowerbound,TOID-upperbound,Li2018,DOID-2021}.

\vspace{.2cm}

\noindent
Given a nontrivial connected graph $G$, the \emph{subdivision graph} $\mathtt{S}(G)$ is obtained from $G$ by subdividing each edge exactly once. Formally,  $V(\mathtt{S}(G))=V(G)\cup V_E(G)$ and $E(\mathtt{S}(G))=\{v_iv^{i,j},v_jv^{i,j}: v^{i,j}\in V_E(G)\}$. By definition, it is easy to check that if $T$ is a tree, then $\mathtt{S}(T)$ is also a tree. In Figure~\ref{fig-1} we show a tree $T$ and its corresponding graph $\mathtt{S}(T)$. For each of these two trees, the set of black vertices
describes a TOIDS of minimum cardinality.

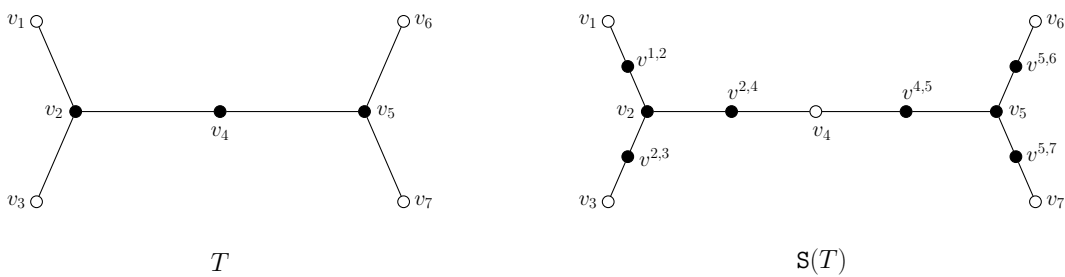
\begin{figure}[h!]
		\centering
		\begin{tikzpicture}[scale=.4, transform shape] 
			\pgfmathsetmacro{\pw}{12}  
			\pgfmathsetmacro{\ph}{3}  
			\pgfmathsetmacro{\tage}{0.6}  
			\pgfmathsetmacro{\ps}{19}  
			\pgfmathsetmacro{\pa}{1.3}  
			\pgfmathsetmacro{\pl}{2}  
			\pgfmathsetmacro{\tx}{0.8}  
			\node[draw, shape=circle, fill = black] (v1) at (\pw/5,1) {};
			\node[draw, shape=circle] (v2) at (\pw/5-\pa,1+\ph) {};
			\node[draw, shape=circle] (v3) at (\pw/5-\pa,1-\ph) {};
			\node[draw, shape=circle, fill = black] (v5) at (3*\pw/5,1) {};
			\node[draw, shape=circle, fill = black] (v7) at (5*\pw/5,1) {};
			\node[draw, shape=circle] (v8) at (5*\pw/5+\pa,1+\ph) {};
			\node[draw, shape=circle] (v9) at (5*\pw/5+\pa,1-\ph) {};
			\draw(v2)--(v1)--(v3);
			\draw(v8)--(v7)--(v9);
			\draw(v1)--(v5)--(v7);	
			\node at (\pw/5-0.9*\pa-\tx,1+\ph) {\LARGE $v_1$};
			\node at (\pw/5-0.9*\tx,1) {\LARGE $v_2$};
			\node at (\pw/5-0.9*\pa-\tx,1-\ph) {\LARGE $v_3$};
			\node at (3*\pw/5,1-0.9*\tx) {\LARGE $v_4$};
			\node at (5*\pw/5+0.9*\tx,1) {\LARGE $v_5$};
			\node at (5*\pw/5+0.9*\pa+\tx,1+\ph) {\LARGE $v_6$};
			\node at (5*\pw/5+0.9*\pa+\tx,1-\ph) {\LARGE $v_7$};
			\node at (3*\pw/5,-1-\ph) {\Huge $T$};
			\node[draw, shape=circle] (w1) at (\pw/5-\pa+\ps,1+\ph) {};
			\node[draw, shape=circle, fill = black] (w2) at (\pw/5+\ps,1) {};
			\node[draw, shape=circle] (w3) at (\pw/5-\pa+\ps,1-\ph) {};
			\node[draw, shape=circle] (w5) at (3*\pw/5+\ps+2*\pl/5,1) {};
			\node[draw, shape=circle, fill = black] (w7) at (5*\pw/5+\ps+\pl,1) {};
			\node[draw, shape=circle] (w8) at (5*\pw/5+\ps+\pl+\pa,1+\ph) {};
			\node[draw, shape=circle] (w9) at (5*\pw/5+\ps+\pl+\pa,1-\ph) {};
			\node[draw, shape=circle, fill = black] (w12) at (\pw/5+\ps-0.5*\pa,1+0.5*\ph) {};
			\node[draw, shape=circle, fill = black] (w23) at (\pw/5-0.5*\pa+\ps,1-0.5*\ph) {};
			\node[draw, shape=circle, fill = black] (w25) at (2*\pw/5+\ps+\pl/5,1) {};
			\node[draw, shape=circle, fill = black] (w57) at (4*\pw/5+\ps+7*\pl/10,1) {};
			\node[draw, shape=circle, fill = black] (w78) at (5*\pw/5+\ps+\pl+0.5*\pa,1+0.5*\ph) {};
			\node[draw, shape=circle, fill = black] (w79) at (5*\pw/5+\ps+\pl+0.5*\pa,1-0.5*\ph) {};
			\draw(w1)--(w12)--(w2)--(w23)--(w3);
			\draw(w8)--(w78)--(w7)--(w79)--(w9);
			\draw(w2)--(w25)--(w5)--(w57)--(w7);		
			\node at (\pw/5-0.9*\pa+\ps-\tx,1+\ph) {\LARGE $v_1$};
			\node at (\pw/5+\ps-0.9*\tx,1) {\LARGE $v_2$};
			\node at (\pw/5-0.9*\pa+\ps-\tx,1-\ph) {\LARGE $v_3$};
			\node at (3*\pw/5+\ps+4*\pl/8,1-0.9*\tx) {\LARGE $v_4$};
			\node at (5*\pw/5+\ps+\pl+0.9*\tx,1) {\LARGE $v_5$};
			\node at (5*\pw/5+\ps+\pl+0.9*\pa+\tx,1+\ph) {\LARGE $v_6$};
			\node at (5*\pw/5+\ps+\pl+0.9*\pa+\tx,1-\ph) {\LARGE $v_7$};
			\node at (\pw/5+\ps-0.5*\pa+\tx,1+0.6*\ph) {\LARGE $v^{1,2}$};
			\node at (\pw/5+0.8*\pa+\ps-\tx,1-0.5*\ph) {\LARGE $v^{2,3}$};
			\node at (2*\pw/5+1.02*\ps+\pl/5,1+0.9*\tx) {\LARGE $v^{2,4}$};
			\node at (4*\pw/5+1.02*\ps+7*\pl/10,1+0.9*\tx) {\LARGE $v^{4,5}$};
			\node at (5*\pw/5+\ps+\pl+0.55*\pa+\tx,1+0.55*\ph) {\LARGE $v^{5,6}$};
			\node at (5*\pw/5+\ps+\pl+0.55*\pa+\tx,1-0.45*\ph) {\LARGE $v^{5,7}$};
			\node at (3*\pw/5+\ps+4*\pl/8,-1-\ph) {\Huge $\mathtt{S}(T)$};
		\end{tikzpicture}
		\caption{A tree $T$, and the corresponding tree $\mathtt{S}(T)$.}\label{fig-1}
	\end{figure}

\noindent
Recently, Cabrera-Mart\'inez et al. \cite{Cabrera-TOID-2026} provided the following result, which gives lower and upper bounds for the total outer-independent domination number of the subdivision of a nontrivial tree.

\begin{theorem}{\rm \cite{Cabrera-TOID-2026}}\label{theo-tree}
For any nontrivial tree $T$, 
$$\frac{4n(T)-l(T)-s(T)}{3}\leq \gamma_{t}^{oi}(\mathtt{S}(T))\leq \frac{4n(T)-l(T)+s(T)-2}{3}.$$
\end{theorem}

\noindent
In this paper, we provide constructive characterizations of the families of trees that attain the bounds given in the previous theorem.

\subsection{Additional notation and terminology, and some useful tools}

 We use the notation $P_r$ to denote the path graph of order $r$. Given a graph $G$, by \emph{attaching} a path $P_r$ to a vertex $v_i\in V(G)$ we mean adding the path $P_r$ and joining $v_i$ by an edge to a leaf of $P_r$. Given a set $S\subset V(G)$, $G-S$ denotes the graph obtained from $G$ by removing all  vertices in $S$ together with all edges incident to them. Next, we define some sets that are relevant to our work.
 
\begin{itemize}
\item $\mathcal{S}_s(G)=\{v_i\in \mathcal{S}(G): |N_G(v_i)\cap \mathcal{L}(G)|\geq 2\}$ is the set of \emph{strong support vertices} of $G$.
\item $\mathcal{L}_s(G)=\{v_i\in \mathcal{L}(G): |N_G(v_i)\cap \mathcal{S}_s(G)|=1\}$ is the set of \emph{strong leaves} of $G$.
\item $\mathcal{L}_w(G)=\mathcal{L}(G)\setminus \mathcal{L}_s(G)$ is the set of \emph{weak leaves} of $G$.
\item $\mathcal{SS}(G)=\{v_i\in V(G)\setminus (\mathcal{S}(G)\cup \mathcal{L}(G)): |N_G(v_i)\cap \mathcal{S}(G)|\geq 1\}$ is the set of \emph{semi-support vertices} of $G$.
\item $\mathcal{N}_{SS}(G)=\{v_i\in V(G)\setminus (\mathcal{SS}(G)\cup \mathcal{S}(G)) : |N_G(v_i)\cap \mathcal{SS}(G)|\ge 1\}$ is the set of vertices adjacent to semi-support vertices of $G$ that are neither semi-support nor support vertices.
\end{itemize}

\noindent
A \emph{tree} $T$ is an  acyclic connected graph. A \emph{rooted tree} $T$ is a tree with a distinguished vertex $v_r$, called the root. Let $v_i\in V(T)\setminus \{v_r\}$. A \emph{descendant} of $v_i$ is a vertex $v_j\neq v_i$ such that the unique $v_r-v_j$ path contains $v_i$. The set of descendants of $v_i$ is denoted by $D(v_i)$. The \emph{maximal subtree} at $v_i$ is the subtree of $T$ induced by $D(v_i)\cup\{v_i\}$ and is denoted by $T_{v_i}$.

\vspace{.2cm}

\noindent
For an integer $r\geq 2$, the tree $Q_r$ is defined as the tree obtained from one copy of $P_1$ and $r$ copies of $P_3$ by joining, with an edge, one leaf of each copy of $P_3$ to the vertex of $P_1$. The unique vertex in $\mathcal{N}_{SS}(Q_r)$ is called the \emph{central vertex} of $Q_r$. In Figure~\ref{fig-Q} we show the tree $Q_3$ and its central vertex $v_1$.

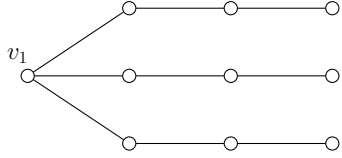
\begin{figure}[h!]
\centering
\begin{tikzpicture}[scale=.45, transform shape] 

\node[draw, shape=circle] (v) at (-1.5,0) {};
\node[draw, shape=circle] (v1) at (1.5,0) {};
\node[draw, shape=circle] (v11) at (4.5,0) {};
\node[draw, shape=circle] (v111) at (7.5,0) {};
\node[draw, shape=circle] (v2) at (1.5,2) {};
\node[draw, shape=circle] (v22) at (4.5,2) {};
\node[draw, shape=circle] (v222) at (7.5,2) {};
\node[draw, shape=circle] (v3) at (1.5,-2) {};
\node[draw, shape=circle] (v33) at (4.5,-2) {};
\node[draw, shape=circle] (v333) at (7.5,-2) {};

\node at (-1.8,0.6) {\LARGE $v_1$};
			
\draw(v)--(v1)--(v11)--(v111);
\draw(v)--(v2)--(v22)--(v222);
\draw(v)--(v3)--(v33)--(v333);		
			
\end{tikzpicture}
\caption{The tree $Q_3$ and its central vertex $v_1$.}\label{fig-Q}
\end{figure}

\noindent
We conclude this section with the following lemmas, which are useful tools in the proofs of the characterizations.

\begin{lemma}{\rm \cite{Cabrera-TOID-2026}}\label{lem-tree}
For any tree $T$ of order at least three, there exists a $\gamma_{t}^{oi}(\mathtt{S}(T))$-set $D$ such that $\mathcal{S}(T)\cup \mathcal{SS}(T)\subseteq D$.  
\end{lemma}

\begin{lemma}\label{lem-tree-attaching}
The following statements hold for any nontrivial tree $T$.
\begin{enumerate}
\item[{\rm (i)}] If $T$ is obtained from any nontrivial tree $T'$  by attaching a path $P_1$ to any vertex in $\mathcal{S}(T')$, then $\gamma_{t}^{oi}(\mathtt{S}(T))=\gamma_{t}^{oi}(\mathtt{S}(T'))+1$.
\item[{\rm (ii)}] If $T$ is obtained from any nontrivial tree $T'$ by attaching a path $P_2$ to any vertex in $ \mathcal{S}(T')\cup \mathcal{SS}(T')$, then $\gamma_{t}^{oi}(\mathtt{S}(T))=\gamma_{t}^{oi}(\mathtt{S}(T'))+2$.
\item[{\rm (iii)}] If $T$ is obtained from any nontrivial tree $T'$ by attaching a path $P_3$ to any vertex in $\mathcal{S}(T')\cup\mathcal{L}_w(T')$, then $\gamma_{t}^{oi}(\mathtt{S}(T))=\gamma_{t}^{oi}(\mathtt{S}(T'))+4$.
\item[{\rm (iv)}] If $T$ is obtained from any nontrivial tree $T'$ by identifying the central vertex of the tree $Q_r$ with a vertex in $\mathcal{L}_w(T')$, then $\gamma_{t}^{oi}(\mathtt{S}(T))=\gamma_{t}^{oi}(\mathtt{S}(T'))+4r$.
\end{enumerate}
\end{lemma}

\begin{proof}
If $n(T')=2$, then $T'\cong P_2$, and it is straightforward to verify that (i), (ii), (iii) and (iv) hold. Hence, we consider from now on that $n(T')\geq 3$.

\vspace{.1cm}

\noindent
First, we assume that $T$ is obtained from $T'$ by adding the vertex $v_1$ and the edge $v_1v_k$, where $v_k\in \mathcal{S}(T')$. By Lemma~\ref{lem-tree}, there exists a $\gamma_{t}^{oi}(\mathtt{S}(T'))$-set $D$ containing vertex $v_k$. Observe that $D\cup \{v^{1,k}\}$ is a TOIDS of $\mathtt{S}(T)$. Hence, $\gamma_{t}^{oi}(\mathtt{S}(T))\leq |D\cup \{v^{1,k}\}|=\gamma_{t}^{oi}(\mathtt{S}(T'))+1$. Let $B$ be a $\gamma_{t}^{oi}(\mathtt{S}(T))$-set which satisfies the condition given in Lemma~\ref{lem-tree}. This implies that $v_k,v^{1,k}\in B$. Since $v_k\in \mathcal{S}(T')$, there exists a vertex $v_l\in N_{T'}(v_k)\cap \mathcal{L}(T')$. Hence, $v^{k,l}\in B$, which implies that $B\setminus \{v^{1,k}\}$ is a TOIDS of $\mathtt{S}(T')$. Thus, $\gamma_{t}^{oi}(\mathtt{S}(T'))\leq |B\setminus \{v^{1,k}\}|=\gamma_{t}^{oi}(\mathtt{S}(T))-1$. Therefore, $\gamma_{t}^{oi}(\mathtt{S}(T))=\gamma_{t}^{oi}(\mathtt{S}(T'))+1$, which completes the proof of (i).

\vspace{.1cm}

\noindent
Next, we assume that $T$ is obtained from $T'$ by adding the path $v_1v_2$ and the edge $v_1v_k$, where $v_k\in \mathcal{S}(T')\cup \mathcal{SS}(T')$. By Lemma~\ref{lem-tree}, there exists a $\gamma_{t}^{oi}(\mathtt{S}(T'))$-set $D'$ containing vertex $v_k$. Observe that $D'\cup \{v_1,v^{1,2}\}$ is a TOIDS of $\mathtt{S}(T)$. Hence, $\gamma_{t}^{oi}(\mathtt{S}(T))\leq |D'\cup \{v_1,v^{1,2}\}|=\gamma_{t}^{oi}(\mathtt{S}(T'))+2$. Let $B'$ be a $\gamma_{t}^{oi}(\mathtt{S}(T))$-set which satisfies the condition given in Lemma~\ref{lem-tree}, and without loss of generality, assume that $|B'\cap \{v_k,v^{1,k},v_1,v^{1,2},v_2\}|$ is minimum. Since $v_1,v_k\in \mathcal{S}(T)\cup \mathcal{SS}(T)$, it follows that $v_1,v_k\in B'$. Moreover, by the minimality of $|B'\cap \{v_k,v^{1,k},v_1,v^{1,2},v_2\}|$, we have that $v^{1,2}\in B'$ and $v^{1,k},v_2\notin B'$. This implies that $B'\setminus \{v_1,v^{1,2}\}$ is a TOIDS of $\mathtt{S}(T')$, and as a consequence, $\gamma_{t}^{oi}(\mathtt{S}(T'))\leq \gamma_{t}^{oi}(\mathtt{S}(T))-2$. Therefore, $\gamma_{t}^{oi}(\mathtt{S}(T))=\gamma_{t}^{oi}(\mathtt{S}(T'))+2$, which completes the proof of (ii).

\vspace{.1cm}

\noindent
Now, assume that $T$ is obtained from $T'$ by adding the path $v_1v_2v_3$ and the edge $v_1v_k$, where $v_k\in \mathcal{S}(T')\cup\mathcal{L}_w(T')$. Note that any $\gamma_{t}^{oi}(\mathtt{S}(T'))$-set $D''$ can be extended to a TOIDS of $\mathtt{S}(T)$ by adding the set $\{v^{1,k},v_1,v_2,v^{2,3}\}$. Hence, $\gamma_{t}^{oi}(\mathtt{S}(T))\leq |D''\cup \{v^{1,k},v_1,v_2,v^{2,3}\}|=\gamma_{t}^{oi}(\mathtt{S}(T'))+4$. Let $B''$ be a $\gamma_{t}^{oi}(\mathtt{S}(T))$-set which satisfies the condition given in Lemma~\ref{lem-tree}, and without loss of generality, assume that $|B''\cap \{v_k,v^{1,k},v_1,v^{1,2},v_2,v^{2,3},v_3\}|$ is minimum. Since $v_1,v_2\in \mathcal{S}(T)\cup \mathcal{SS}(T)$, it follows that $v_1,v_2\in B''$, which leads to $v^{2,3},v^{1,k}\in B''$ and $v^{1,2},v_3\notin B''$. 
By the minimality of $|B''\cap \{v_k,v^{1,k},v_1,v^{1,2},v_2,v^{2,3},v_3\}|$ and the fact that $v_k\in \mathcal{S}(T')\cup\mathcal{L}_w(T')$, there exists a vertex $v_l\in N_T(v_k)\setminus \{v_1\}$ such that $v^{k,l}\in B''$. This implies that $B''\setminus \{v^{1,k},v_1,v_2,v^{2,3}\}$ is a TOIDS of $\mathtt{S}(T')$, and as a consequence, $\gamma_{t}^{oi}(\mathtt{S}(T'))\leq\gamma_{t}^{oi}(\mathtt{S}(T))-4$. Therefore, $\gamma_{t}^{oi}(\mathtt{S}(T))=\gamma_{t}^{oi}(\mathtt{S}(T'))+4$, which completes the proof of (iii).

\vspace{.1cm}

\noindent
Finally, assume that $T$ is obtained from $T'$ by identifying the central vertex $v_1$ of the tree $Q_r$ with a vertex $v_k\in \mathcal{L}_w(T')$. Let $v_l$ be the support vertex adjacent to $v_k$ in $T'$, that is, $N_{T'}(v_k)=\{v_l\}$. To facilitate the proof, we first define the structure of $Q_r$. Let $V(Q_r)=\{v_1,\ldots, v_{3r+1}\}$, where $\mathcal{N}_{SS}(Q_r)=\{v_1\}$ and
\begin{itemize}
\item $\mathcal{SS}(Q_r)=\bigcup_{i\in \{1,\ldots,r\}}\{v_{3i-1}\}, \hspace{.2cm} \mathcal{S}(Q_r)=\bigcup_{i\in \{1,\ldots,r\}}\{v_{3i}\} \hspace{.2cm} \text{and} \hspace{.2cm} \mathcal{L}(Q_r)=\bigcup_{i\in \{1,\ldots,r\}}\{v_{3i+1}\}$.
\item $E(Q_r)=\bigcup_{i\in \{1,\ldots, r\}}\{v_1v_{3i-1}, v_{3i-1}v_{3i}, v_{3i}v_{3i+1}\}$.  
\end{itemize}
Let $A=\mathcal{S}(Q_r)\cup \mathcal{SS}(Q_r)\cup \left(\cup_{i\in \{1,\ldots,r\}}\{v^{1,3i-1},v^{3i,3i+1}\}\right)$. Observe that any $\gamma_{t}^{oi}(\mathtt{S}(T'))$-set~$D'''$ can be extended to a TOIDS of $\mathtt{S}(T)$ by adding the set $A$.
Hence, $\gamma_{t}^{oi}(\mathtt{S}(T))\leq |D'''\cup A|=\gamma_{t}^{oi}(\mathtt{S}(T'))+4r$. Now, let $B'''$ be a $\gamma_{t}^{oi}(\mathtt{S}(T))$-set which satisfies the condition given in Lemma~\ref{lem-tree}, and without loss of generality, assume that $|B'''\cap V(Q_r)|$ is minimum. It is straightforward to verify that $A\subseteq B'''$. By the minimality of $|B'''\cap V(Q_r)|$ and the fact that $v_k\in \mathcal{L}_w(T')$, it follows that $v^{k,l}\in B'''$. This implies that $B'''\setminus A$ is a TOIDS of $\mathtt{S}(T')$, and as a consequence, $\gamma_{t}^{oi}(\mathtt{S}(T'))\leq|B'''\setminus A|=\gamma_{t}^{oi}(\mathtt{S}(T))-4r$. Therefore, $\gamma_{t}^{oi}(\mathtt{S}(T))=\gamma_{t}^{oi}(\mathtt{S}(T'))+4r$, which completes the proof of (iv). 
\end{proof}

\begin{lemma}\label{lem-equality-upper-bound}
Let $T$ be a tree different from $P_5$ such that $\gamma_{t}^{oi}(\mathtt{S}(T))=\frac{4n(T)-l(T)+s(T)-2}{3}$. If $\mathcal{L}_s(T)=\emptyset$, then $|N_T(v_k)\cap \mathcal{S}(T)|=1$ for every vertex $v_k\in \mathcal{SS}(T)$.
\end{lemma}

\begin{proof}
Let $v_k\in \mathcal{SS}(T)$. By definition, we have that $N_T(v_k)\cap \mathcal{S}(T)\neq \emptyset$. Suppose that there exist $v_1,v_2\in N_T(v_k)\cap \mathcal{S}(T)$. Since $\mathcal{L}_s(T)=\emptyset$, we define $N_T(v_i)\cap \mathcal{L}(T)=\{v_{i+2}\}$ for $i\in \{1,2\}$. Let $T'=T-\{v_1,v_3\}$. Observe that $T'\not\cong P_3$, since $T\not\cong P_5$. Hence, $n(T)=n(T')+2$, $l(T)=l(T')+1$ and $s(T)=s(T')+1$. Now, let $D$ be a $\gamma_{t}^{oi}(\mathtt{S}(T'))$-set which satisfies the condition given in Lemma~\ref{lem-tree}. Since $v_k\in \mathcal{SS}(T')$, it follows that $v_k\in D$. As a consequence, $D\cup \{v_1,v^{1,3}\}$ is a TOIDS of $\mathtt{S}(T)$, which implies that $\gamma_{t}^{oi}(\mathtt{S}(T))\leq \gamma_{t}^{oi}(\mathtt{S}(T'))+2$. Therefore,
\begin{equation*}
\begin{aligned}
\gamma_{t}^{oi}(\mathtt{S}(T'))\geq \gamma_{t}^{oi}(\mathtt{S}(T))-2
&=\tfrac{4n(T)-l(T)+s(T)-2}{3}-2\\[5pt]
&=\tfrac{4(n(T')+2)-(l(T')+1)+(s(T')+1)-2}{3}-2>\tfrac{4n(T')-l(T')+s(T')-2}{3},
\end{aligned}
\end{equation*} 
which contradicts the upper bound given in Theorem~\ref{theo-tree}. Hence, $|N_T(v_k)\cap \mathcal{S}(T)|=1$, which completes the proof.
\end{proof}

\section{Trees $T$ with $\gamma_{t}^{oi}(\mathtt{S}(T))=\frac{4n(T)-l(T)-s(T)}{3}$}

To characterize the trees attaining the lower bound given in Theorem~\ref{theo-tree}, we introduce the following family of trees. Let $\mathcal{F}$ be the family of trees $T$ that can be obtained from a sequence of trees $T_0,\ldots, T_k=T$, with $k\geq 0$ and $T_0=P_2$. If $k\geq 1$, then for each $i\in \{1,\ldots, k\}$, the tree $T_i$ can be obtained from the tree $T'=T_{i-1}$ by one of the following three operations defined below.
\begin{description}
  \item[Operation $F_1$:]  Attach a path $P_1$ to a vertex  in $\mathcal{S}(T')$.
  \item[Operation $F_2$:]  Attach a path $P_2$ to a vertex  in $\mathcal{S}(T')\cup \mathcal{SS}(T')$.
  \item[Operation $F_3$:]  Attach a path $P_3$ to a vertex  in $\mathcal{L}_w(T')$.
\end{description}

\vspace{.2cm}

\noindent
We first show that every tree $T$ in the family $\mathcal{F}$ satisfies that $\gamma_{t}^{oi}(\mathtt{S}(T))=\frac{4n(T)-l(T)-s(T)}{3}$.

\begin{lemma}\label{lem-lower-bound-right}
Let $T$ be a nontrivial tree. If $T \in \mathcal{F}$, then $\gamma_{t}^{oi}(\mathtt{S}(T))=\frac{4n(T)-l(T)-s(T)}{3}$.
\end{lemma}

\begin{proof}
Let $T$ be a tree belonging to the family $\mathcal{F}$. We proceed to prove that $\gamma_{t}^{oi}(\mathtt{S}(T))=(4n(T)-l(T)-s(T))/3$ by induction on the order of $T$. If $n(T)=2$, then $T=P_2$ and $\gamma_{t}^{oi}(\mathtt{S}(T))=2=(4n(T)-l(T)-s(T))/3$, as required. This particular case establishes the base case. Suppose that $n(T)\geq 3$ and that every tree $T^*$ in $\mathcal{F}$, with $2\leq n(T^*)< n(T)$, satisfies that $\gamma_{t}^{oi}(\mathtt{S}(T^*))=(4n(T^*)-l(T^*)-s(T^*))/3$. Since $T\in \mathcal{F}$, it is clear that $T$ can be obtained from a sequence of trees $T_0,\ldots, T_k=T$, with $T_0=P_2$ and $k\geq 1$. Let $T'=T_{k-1}$, which implies that $T'\in \mathcal{F}$, and by induction hypothesis, it follows that $\gamma_{t}^{oi}(\mathtt{S}(T'))=(4n(T')-l(T')-s(T'))/3$. We consider the following three cases, depending on which operation is used to obtain the tree $T$ from $T'$.

\vspace{.2cm}

\noindent
\textbf{Case 1:} $T$ is obtained from $T'$ by Operation $F_1$. In this case, $T$ is formed by adding a vertex $v_1$ and the edge $v_1v_i$ where $v_i\in \mathcal{S}(T')$. By Lemma~\ref{lem-tree-attaching}-(i), the induction hypothesis, and the relations $n(T')=n(T)-1$, $s(T')=s(T)$, and $l(T')=l(T)-1$, we obtain the following desired equality:
{\small
\begin{equation*}
 \gamma_{t}^{oi}(\mathtt{S}(T))=\gamma_{t}^{oi}(\mathtt{S}(T'))+1=\tfrac{4n(T')-l(T')-s(T')}{3}+1=\tfrac{4(n(T)-1)-(l(T)-1)-s(T)}{3}+1= \tfrac{4n(T)-l(T)-s(T)}{3}.
\end{equation*}
}

\vspace{.3cm}

\noindent
\textbf{Case 2:} $T$ is obtained from $T'$ by Operation $F_2$. In this case, $T$ is formed by adding the path $v_1v_2$ and the edge $v_1v_i$ where $v_i\in \mathcal{S}(T')\cup \mathcal{SS}(T')$. By Lemma~\ref{lem-tree-attaching}-(ii), the induction hypothesis, and the relations $n(T')=n(T)-2$, $s(T')=s(T)-1$, and $l(T')=l(T)-1$, we obtain the following desired equality:
{\small
\begin{equation*}
\gamma_{t}^{oi}(\mathtt{S}(T))=\gamma_{t}^{oi}(\mathtt{S}(T'))+2=\tfrac{4n(T')-l(T')-s(T')}{3}+2=\tfrac{4(n(T)-2)-(l(T)-1)-(s(T)-1)}{3}+2= \tfrac{4n(T)-l(T)-s(T)}{3}.
\end{equation*}
}

\vspace{.3cm}

\noindent
\textbf{Case 3:} $T$ is obtained from $T'$ by Operation $F_3$. In this case, $T$ is formed by adding the path $v_1v_2v_3$ and the edge $v_1v_i$ where $v_i\in \mathcal{L}_w(T')$. By Lemma~\ref{lem-tree-attaching}-(iii), the induction hypothesis, and the relations $n(T')=n(T)-3$, $s(T')=s(T)$, and $l(T')=l(T)$, we obtain the following desired equality:
\begin{equation*}
\gamma_{t}^{oi}(\mathtt{S}(T))=\gamma_{t}^{oi}(\mathtt{S}(T'))+4=\tfrac{4n(T')-l(T')-s(T')}{3}+4=\tfrac{4(n(T)-3)-l(T)-s(T)}{3}+4= \tfrac{4n(T)-l(T)-s(T)}{3}.
\end{equation*}

\vspace{.3cm}

\noindent
Therefore, and as a consequence of the three cases above, the proof is complete. 
\end{proof}
 
\vspace{.2cm}

\noindent
We next show that any tree $T$ with $\gamma_{t}^{oi}(\mathtt{S}(T))=(4n(T)-l(T)-s(T))/3$ belongs to the family~$\mathcal{F}$.

\begin{lemma}\label{lem-lower-bound-left}
Let $T$ be a nontrivial tree. If $\gamma_{t}^{oi}(\mathtt{S}(T))=\frac{4n(T)-l(T)-s(T)}{3}$, then $T \in \mathcal{F}$.
\end{lemma}

\begin{proof}
Let $T$ be a tree with $\gamma_{t}^{oi}(\mathtt{S}(T))=\frac{4n(T)-l(T)-s(T)}{3}$. We proceed to prove that $T \in \mathcal{F}$ by induction on the order of $T$. If $n(T)=2$, then $T=P_2$, which belongs to $\mathcal{F}$. Suppose that $n(T)\geq 3$ and that every tree $T^*$ with $\gamma_{t}^{oi}(\mathtt{S}(T^*))=(4n(T^*)-l(T^*)-s(T^*))/3$ and $2\leq n(T^*)< n(T)$ satisfies that $T^*\in \mathcal{F}$. Let $v_1v_2\cdots v_dv_{d+1}$ be a diametral path in $T$ ($d$ represents the diameter of $T$). We now proceed with the following claims.

\vspace{.25cm}

\noindent
{\bf Claim I:} If $|N_T(v_{2})|\geq 3$, then $T\in \mathcal{F}$.

\vspace{.2cm}

\noindent
{\em Proof of Claim I.} Let $T'=T-\{v_1\}$. Observe that $v_2\in \mathcal{S}(T')$. By Lemma~\ref{lem-tree-attaching}-(i), it follows that $\gamma_{t}^{oi}(\mathtt{S}(T))=\gamma_{t}^{oi}(\mathtt{S}(T'))+1$. Using the hypothesis of the lemma and the relations $n(T)=n(T')+1$, $s(T)=s(T')$, and $l(T)=l(T')+1$, we obtain that
\begin{equation*}
\begin{aligned}
\gamma_{t}^{oi}(\mathtt{S}(T'))=\gamma_{t}^{oi}(\mathtt{S}(T))-1&=\tfrac{4n(T)-l(T)-s(T)}{3}-1\\[5pt]
&=\tfrac{4(n(T')+1)-(l(T')+1)-s(T')}{3}-1=\tfrac{4n(T')-l(T')-s(T')}{3}.
\end{aligned}
\end{equation*} 

\noindent
Hence, by the induction hypothesis, it follows that $T'\in \mathcal{F}$. Therefore, $T$ can be obtained from $T'$ by Operation $F_1$, and consequently, $T\in \mathcal{F}$.

\vspace{.4cm}

\noindent
{\bf Claim II:} If $|N_T(v_{2})|=2$ and $|N_T(v_{3})|\geq 3$, then $T\in \mathcal{F}$.

\vspace{.2cm}

\noindent
{\em Proof of Claim II.} Let $T'=T-\{v_1,v_2\}$. Observe that $v_3\in \mathcal{S}(T')\cup \mathcal{SS}(T')$. By Lemma~\ref{lem-tree-attaching}-(ii), it follows that $\gamma_{t}^{oi}(\mathtt{S}(T))=\gamma_{t}^{oi}(\mathtt{S}(T'))+2$. Using the hypothesis of the lemma and the relations $n(T)=n(T')+2$, $s(T)=s(T')+1$, and $l(T)=l(T')+1$, we obtain that
\begin{equation*}
\begin{aligned}
\gamma_{t}^{oi}(\mathtt{S}(T'))=\gamma_{t}^{oi}(\mathtt{S}(T))-2&=\tfrac{4n(T)-l(T)-s(T)}{3}-2\\[5pt]
&=\tfrac{4(n(T')+2)-(l(T')+1)-(s(T')+1)}{3}-2=\tfrac{4n(T')-l(T')-s(T')}{3}.
\end{aligned}
\end{equation*} 

\noindent
Hence, by the induction hypothesis, it follows that $T'\in \mathcal{F}$. Therefore, $T$ can be obtained from $T'$ by Operation $F_2$, and consequently, $T\in \mathcal{F}$.

\vspace{.4cm}

\noindent
{\bf Claim III:} If $|N_T(v_{2})|=|N_T(v_{3})|=2$, then $T\in \mathcal{F}$.

\vspace{.2cm}

\noindent
{\em Proof of Claim III.} Let $T'=T-\{v_1,v_2,v_3\}$. Since $T\not\cong P_4$, it follows that $n(T')\geq 2$. Let $D$ be a $\gamma_{t}^{oi}(\mathtt{S}(T))$-set which satisfies the condition given in Lemma~\ref{lem-tree}, and without loss of generality, assume that $|D\cap \{v_4,v^{3,4},v_3,v^{2,3},v_2,v^{1,2},v_1\}|$ is minimum. Since $v_2,v_3\in \mathcal{S}(T)\cup \mathcal{SS}(T)$, it follows that $v_2,v_3\in D$. Moreover, by the minimality of $|D\cap \{v_4,v^{3,4},v_3,v^{2,3},v_2,v^{1,2},v_1\}|$, we have that $v^{1,2},v^{3,4}\in D$ and $v_1,v^{2,3}\notin D$. Now, suppose that $v_4\in V(T')\setminus \mathcal{L}_w(T')$, and let us analyze the following two cases.

\vspace{.2cm}

\noindent
\textbf{Case 1:} $v_4\in\mathcal{L}_s(T')$. In this case, it is straightforward to deduce that $v_4\notin D$. As a consequence, $D\setminus \{v^{3,4},v_3,v_2,v^{1,2}\}$ is a TOIDS of $\mathtt{S}(T')$, which implies that $\gamma_{t}^{oi}(\mathtt{S}(T'))\leq\gamma_{t}^{oi}(\mathtt{S}(T))-4$. By the hypothesis of the lemma and the relations $n(T)=n(T')+3$, $s(T)=s(T')+1$, and $l(T)=l(T')$, we obtain that
\begin{equation*}
\begin{aligned}
\gamma_{t}^{oi}(\mathtt{S}(T'))\leq\gamma_{t}^{oi}(\mathtt{S}(T))-4&=\tfrac{4n(T)-l(T)-s(T)}{3}-4\\[5pt]
&=\tfrac{4(n(T')+3)-l(T')-(s(T')+1)}{3}-4<\tfrac{4n(T')-l(T')-s(T')}{3},
\end{aligned}
\end{equation*}
which contradicts the lower bound given in Theorem~\ref{theo-tree}.

\vspace{.2cm}

\noindent
\textbf{Case 2:} $v_4\in V(T')\setminus \mathcal{L}(T')$. If $v_4\notin D$, then $D\setminus \{v^{3,4},v_3,v_2,v^{1,2}\}$ is a TOIDS of $\mathtt{S}(T')$, which implies that $\gamma_{t}^{oi}(\mathtt{S}(T'))\leq\gamma_{t}^{oi}(\mathtt{S}(T))-4$. By the hypothesis of the lemma and the relations $n(T)=n(T')+3$, $s(T)=s(T')+1$, and $l(T)=l(T')+1$, we obtain that
\begin{equation*}
\begin{aligned}
\gamma_{t}^{oi}(\mathtt{S}(T'))\leq\gamma_{t}^{oi}(\mathtt{S}(T))-4&=\tfrac{4n(T)-l(T)-s(T)}{3}-4\\[5pt]
&=\tfrac{4(n(T')+3)-(l(T')+1)-(s(T')+1)}{3}-4<\tfrac{4n(T')-l(T')-s(T')}{3},
\end{aligned}
\end{equation*}
which contradicts the lower bound given in Theorem~\ref{theo-tree}. Now, assume that $v_4\in D$. Let $T''=T-\{v_1,v_2\}$. Observe that $D\setminus \{v_3,v_2,v^{1,2}\}$ is a TOIDS of $\mathtt{S}(T'')$, which implies that $\gamma_{t}^{oi}(\mathtt{S}(T''))\leq\gamma_{t}^{oi}(\mathtt{S}(T))-3$. By the hypothesis of the lemma and the relations $n(T)=n(T'')+2$, $s(T)\geq s(T'')$, and $l(T)=l(T'')$, we obtain that
\begin{equation*}
\begin{aligned}
\gamma_{t}^{oi}(\mathtt{S}(T''))\leq\gamma_{t}^{oi}(\mathtt{S}(T))-3&=\tfrac{4n(T)-l(T)-s(T)}{3}-3\\[5pt]
&\leq\tfrac{4(n(T'')+2)-l(T'')-s(T'')}{3}-3<\tfrac{4n(T'')-l(T'')-s(T'')}{3},
\end{aligned}
\end{equation*}

\noindent
which again contradicts the lower bound given in Theorem~\ref{theo-tree}.

\vspace{.2cm}

\noindent
From the contradictions obtained in the two previous cases, we conclude that $v_4\in \mathcal{L}_w(T')$. By Lemma~\ref{lem-tree-attaching}-(iii), it follows that $\gamma_{t}^{oi}(\mathtt{S}(T))=\gamma_{t}^{oi}(\mathtt{S}(T'))+4$. Using the hypothesis of the lemma and the relations $n(T)=n(T')+3$, $s(T)=s(T')$, and $l(T)=l(T')$, we obtain that
\begin{equation*}
\gamma_{t}^{oi}(\mathtt{S}(T'))=\gamma_{t}^{oi}(\mathtt{S}(T))-4=\tfrac{4n(T)-l(T)-s(T)}{3}-4=\tfrac{4(n(T')+3)-l(T')-s(T')}{3}-4=\tfrac{4n(T')-l(T')-s(T')}{3}.
\end{equation*} 

\noindent
Hence, by the induction hypothesis, it follows that $T'\in \mathcal{F}$. Therefore, $T$ can be obtained from $T'$ by Operation $F_3$, and consequently, $T\in \mathcal{F}$. 
\end{proof}

\vspace{.2cm}

\noindent
As an immediate consequence of Lemmas~\ref{lem-lower-bound-right} and \ref{lem-lower-bound-left} we have the desired characterization.

\begin{theorem}
Let $T$ be a nontrivial tree. Then $\gamma_{t}^{oi}(\mathtt{S}(T))=\frac{4n(T)-l(T)-s(T)}{3}$ if and only if  $T\in \mathcal{F}$.
\end{theorem}

\section{Trees $T$ with $\gamma_{t}^{oi}(\mathtt{S}(T))=\frac{4n(T)-l(T)+s(T)-2}{3}$}

To characterize the trees attaining the upper bound given in Theorem~\ref{theo-tree}, we introduce the following family of trees. Let $\mathcal{T}$ be the family of trees $T$ that can be obtained from a sequence of trees $T_0,\ldots, T_k=T$, with $k\geq 0$ and $T_0=P_2$. If $k\geq 1$, then for each $i\in \{1,\ldots, k\}$, the tree $T_i$ can be obtained from the tree $T'=T_{i-1}$ by one of the following three operations defined below.
\begin{description}
\item[Operation $O_1$:]  Attach a path $P_1$ to a vertex  in $\mathcal{S}(T')$.
\item[Operation $O_2$:]  Attach a path $P_3$ to a vertex  in $\mathcal{S}(T')\cup\mathcal{L}_w(T')$.
\item[Operation $O_3$:]  Add a tree $Q_r$ and identify its central vertex with a vertex in $\mathcal{L}_w(T')$.
\end{description}

\vspace{.2cm}

\noindent
We first show that every tree $T$ in the family $\mathcal{T}$ satisfies that $\gamma_{t}^{oi}(\mathtt{S}(T))=\frac{4n(T)-l(T)+s(T)-2}{3}$.

\begin{lemma}\label{lem-upper-bound-right}
Let $T$ be a nontrivial tree. If $T \in \mathcal{T}$, then $\gamma_{t}^{oi}(\mathtt{S}(T))=\frac{4n(T)-l(T)+s(T)-2}{3}$.
\end{lemma}

\begin{proof}
Let $T$ be a tree belonging to the family $\mathcal{T}$. We proceed to prove that $\gamma_{t}^{oi}(\mathtt{S}(T))=(4n(T)-l(T)+s(T)-2)/3$ by induction on the order of $T$. If $n(T)=2$, then $T=P_2$ and $\gamma_{t}^{oi}(\mathtt{S}(T))=2=(4n(T)-l(T)+s(T)-2)/3$, as required. This particular case establishes the base case. Suppose that $n(T)\geq 3$ and that every tree $T^*$ in $\mathcal{T}$, with $2\leq n(T^*)< n(T)$, satisfies that $\gamma_{t}^{oi}(\mathtt{S}(T^*))=(4n(T^*)-l(T^*)+s(T^*)-2)/3$. Since $T\in \mathcal{T}$, it is clear that $T$ can be obtained from a sequence of trees $T_0,\ldots, T_k=T$, with $T_0=P_2$ and $k\geq 1$. Let $T'=T_{k-1}$, which implies that $T'\in \mathcal{T}$, and by induction hypothesis, it follows that $\gamma_{t}^{oi}(\mathtt{S}(T'))=(4n(T')-l(T')+s(T')-2)/3$. We consider the following three cases, depending on which operation is used to obtain the tree $T$ from $T'$.

\vspace{.3cm}

\noindent
\textbf{Case 1:} $T$ is obtained from $T'$ by Operation $O_1$. In this case, $T$ is formed by adding a vertex $v_1$ and the edge $v_1v_i$ where $v_i\in \mathcal{S}(T')$. By Lemma~\ref{lem-tree-attaching}-(i), the induction hypothesis, and the relations $n(T')=n(T)-1$, $s(T')=s(T)$, and $l(T')=l(T)-1$, we obtain the following desired equality:
\begin{equation*}
\begin{aligned}
 \gamma_{t}^{oi}(\mathtt{S}(T))=\gamma_{t}^{oi}(\mathtt{S}(T'))+1&=\tfrac{4n(T')-l(T')+s(T')-2}{3}+1\\[5pt]
&=\tfrac{4(n(T)-1)-(l(T)-1)+s(T)-2}{3}+1= \tfrac{4n(T)-l(T)+s(T)-2}{3}.
\end{aligned}
\end{equation*}

\vspace{.3cm}

\noindent
\textbf{Case 2:} $T$ is obtained from $T'$ by Operation $O_2$. In this case, $T$ is formed by adding the path $v_1v_2v_3$ and the edge $v_1v_i$ where $v_i\in \mathcal{S}(T')\cup\mathcal{L}_w(T')$. If $v_i\in\mathcal{L}_w(T')$, then $n(T')=n(T)-3$, $s(T')=s(T)$, and $l(T')=l(T)$. Hence, by Lemma~\ref{lem-tree-attaching}-(iii) and the induction hypothesis, it follows that
\begin{equation*}
\begin{aligned}
\gamma_{t}^{oi}(\mathtt{S}(T))=\gamma_{t}^{oi}(\mathtt{S}(T'))+4&=\tfrac{4n(T')-l(T')+s(T')-2}{3}+4\\[5pt]
&=\tfrac{4(n(T)-3)-l(T)+s(T)-2}{3}+4= \tfrac{4n(T)-l(T)+s(T)-2}{3}.
\end{aligned}
\end{equation*}

\noindent
Finally, if $v_i\in\mathcal{S}(T')$, then $n(T')=n(T)-3$, $s(T')=s(T)-1$, and $l(T')=l(T)-1$. Hence, by Lemma~\ref{lem-tree-attaching}-(iii) and the induction hypothesis, we obtain the following desired equality:
\begin{equation*}
\begin{aligned}
\gamma_{t}^{oi}(\mathtt{S}(T))
=\gamma_{t}^{oi}(\mathtt{S}(T'))+4
&=\tfrac{4n(T')-l(T')+s(T')-2}{3}+4 \\[5pt]
&=\tfrac{4(n(T)-3)-(l(T)-1)+(s(T)-1)-2}{3}+4
=\tfrac{4n(T)-l(T)+s(T)-2}{3}.
\end{aligned}
\end{equation*}

\vspace{.3cm}

\noindent
\textbf{Case 3:} $T$ is obtained from $T'$ by Operation $O_3$. In this case, $T$ is formed by identifying the central vertex of $Q_r$ with a vertex $v_i\in \mathcal{L}_w(T')$. By Lemma~\ref{lem-tree-attaching}-(iv), the induction hypothesis, and the relations $n(T')=n(T)-3r$, $s(T')=s(T)-r+1$, and $l(T')=l(T)-r+1$, we obtain the following desired equality:
\begin{equation*}
\begin{aligned}
 \gamma_{t}^{oi}(\mathtt{S}(T))=\gamma_{t}^{oi}(\mathtt{S}(T'))+4r&=\tfrac{4n(T')-l(T')+s(T')-2}{3}+4r\\[5pt]
&=\tfrac{4(n(T)-3r)-(l(T)-r+1)+(s(T)-r+1)-2}{3}+4r= \tfrac{4n(T)-l(T)+s(T)-2}{3}.
\end{aligned}
\end{equation*}

\vspace{.2cm}

\noindent
Therefore, and as a consequence of the three cases above, the proof is complete. 
\end{proof}

\vspace{.2cm}

\noindent
We next show that any tree $T$ with $\gamma_{t}^{oi}(\mathtt{S}(T))=(4n(T)-l(T)+s(T)-2)/3$ belongs to the family~$\mathcal{T}$.

\begin{lemma}\label{lem-upper-bound-left}
Let $T$ be a nontrivial tree. If $\gamma_{t}^{oi}(\mathtt{S}(T))=\frac{4n(T)-l(T)+s(T)-2}{3}$, then $T \in \mathcal{T}$.
\end{lemma}

\begin{proof}
Let $T$ be a tree with $\gamma_{t}^{oi}(\mathtt{S}(T))=\frac{4n(T)-l(T)+s(T)-2}{3}$. We proceed to prove that $T \in \mathcal{T}$ by induction on the order of $T$. If $n(T)=2$, then $T=P_2$, which belongs to $\mathcal{T}$. Suppose that $n(T)\geq 3$ and that every tree $T^*$ with $\gamma_{t}^{oi}(\mathtt{S}(T^*))=(4n(T^*)-l(T^*)+s(T^*)-2)/3$ and $2\leq n(T^*)< n(T)$ satisfies that $T^*\in \mathcal{T}$. Let $v_1v_2\cdots v_dv_{d+1}$ be a diametral path in $T$ ($d$ represents the diameter of $T$). From now on we assume that $T$ is a rooted tree with root $v_{d+1}$. We now proceed with the following claims.

\vspace{.2cm}

\noindent
{\bf Claim I:} If $\mathcal{L}_s(T)\neq \emptyset$, then $T\in \mathcal{T}$.

\vspace{.15cm}

\noindent
{\em Proof of Claim I.} Let $v_k\in \mathcal{L}_s(T)$ and let $v_l$ be its corresponding support vertex, that is, $N_T(v_k)=\{v_l\}$. Define $T'=T-\{v_k\}$. Observe that $v_l\in \mathcal{S}(T')$. By Lemma~\ref{lem-tree-attaching}-(i), it follows that $\gamma_{t}^{oi}(\mathtt{S}(T))=\gamma_{t}^{oi}(\mathtt{S}(T'))+1$. Using the hypothesis of the lemma and the relations $n(T)=n(T')+1$, $s(T)=s(T')$, and $l(T)=l(T')+1$, we obtain that
\begin{equation*}
\begin{aligned}
\gamma_{t}^{oi}(\mathtt{S}(T'))=\gamma_{t}^{oi}(\mathtt{S}(T))-1
&=\tfrac{4n(T)-l(T)+s(T)-2}{3}-1\\[5pt]
&=\tfrac{4(n(T')+1)-(l(T')+1)+s(T')-2}{3}-1=\tfrac{4n(T')-l(T')+s(T')-2}{3}.
\end{aligned}
\end{equation*} 

\noindent
Hence, by the induction hypothesis, it follows that $T'\in \mathcal{T}$. Therefore, $T$ can be obtained from $T'$ by Operation $O_1$, and consequently, $T\in \mathcal{T}$.

\vspace{.2cm}

\noindent
{\bf Claim II:} If $\mathcal{L}_s(T)=\emptyset$, then $T\in \mathcal{T}$.

\vspace{.15cm}

\noindent
{\em Proof of Claim II.} Observe that $|N_T(v_{2})|=2$. We first suppose that $|N_T(v_{3})|\geq 3$. Let $T''=T-\{v_1,v_2\}$. Since $v_3\in \mathcal{S}(T'')\cup \mathcal{SS}(T'')$, it follows by Lemma~\ref{lem-tree-attaching}-(ii) that $\gamma_{t}^{oi}(\mathtt{S}(T))=\gamma_{t}^{oi}(\mathtt{S}(T''))+2$.  By the hypothesis of the lemma and the relations $n(T)=n(T'')+2$, $s(T)=s(T'')+1$, and $l(T)=l(T'')+1$, we obtain that
\begin{equation*}
\begin{aligned}
\gamma_{t}^{oi}(\mathtt{S}(T''))=\gamma_{t}^{oi}(\mathtt{S}(T))-2
&=\tfrac{4n(T)-l(T)+s(T)-2}{3}-2\\[5pt]
&=\tfrac{4(n(T'')+2)-(l(T'')+1)+(s(T'')+1)-2}{3}-2>\tfrac{4n(T'')-l(T'')+s(T'')-2}{3},
\end{aligned}
\end{equation*} 
which contradicts the upper bound given in Theorem~\ref{theo-tree}. Hence, $|N_T(v_{3})|=2$. Now, let $T'=T-\{v_1,v_2,v_3\}$. It is straightforward to verify that $n(T')\geq 2$. 
Notice that any $\gamma_{t}^{oi}(\mathtt{S}(T'))$-set $D'$ can be extended to a TOIDS of $\mathtt{S}(T)$ by adding the set $\{v^{1,2},v_2,v_3,v^{3,4}\}$. Hence, $\gamma_{t}^{oi}(\mathtt{S}(T))\leq |D'\cup \{v^{1,2},v_2,v_3,v^{3,4}\}|=\gamma_{t}^{oi}(\mathtt{S}(T'))+4$. 
In addition, observe that $v_4\in \mathcal{N}_{SS}(T')\cup \mathcal{SS}(T')\cup\mathcal{S}(T')\cup\mathcal{L}(T')$. Next, we analyze the following cases.

\vspace{.1cm}

\noindent
\textbf{Case 1:}  $v_4\in \mathcal{S}(T')\cup\mathcal{L}(T')$. We first observe that if $v_4\in \mathcal{L}_s(T')$, then $n(T)=n(T')+3$, $s(T)=s(T')+1$, and $l(T)=l(T')$. Hence,
\begin{equation*}
\begin{aligned}
\gamma_{t}^{oi}(\mathtt{S}(T'))\geq\gamma_{t}^{oi}(\mathtt{S}(T))-4
&=\tfrac{4n(T)-l(T)+s(T)-2}{3}-4\\[5pt]
&=\tfrac{4(n(T')+3)-l(T')+(s(T')+1)-2}{3}-4>\tfrac{4n(T')-l(T')+s(T')-2}{3},
\end{aligned}
\end{equation*} 
which contradicts the upper bound given in Theorem~\ref{theo-tree}.
Therefore,  $v_4\in \mathcal{S}(T')\cup\mathcal{L}_w(T')$. By Lemma~\ref{lem-tree-attaching}-(iii) we have that $\gamma_{t}^{oi}(\mathtt{S}(T))=\gamma_{t}^{oi}(\mathtt{S}(T'))+4$. Now, let us consider the following two subcases.

\vspace{.15cm}

\noindent
\textbf{Subase 1.1:}  $v_4\in \mathcal{L}_w(T')$. In this subcase, we have that $n(T)=n(T')+3$, $s(T)=s(T')$, and $l(T)=l(T')$. By the hypothesis of the lemma and the previous equalities, we obtain that
\begin{equation*}
	\begin{aligned}
		\gamma_{t}^{oi}(\mathtt{S}(T'))=\gamma_{t}^{oi}(\mathtt{S}(T))-4
		&=\tfrac{4n(T)-l(T)+s(T)-2}{3}-4\\[5pt]
		&=\tfrac{4(n(T')+3)-l(T')+s(T')-2}{3}-4=\tfrac{4n(T')-l(T')+s(T')-2}{3}.
	\end{aligned}
\end{equation*} 

\vspace{.15cm}

\noindent
\textbf{Subcase 1.2:}  $v_4\in\mathcal{S}(T')$. In this subcase, we have that $n(T)=n(T')+3$, $s(T)=s(T')+1$, and $l(T)=l(T')+1$. By the hypothesis of the lemma and the previous equalities, we obtain that
\begin{equation*}
	\begin{aligned}
		\gamma_{t}^{oi}(\mathtt{S}(T'))=\gamma_{t}^{oi}(\mathtt{S}(T))-4
		&=\tfrac{4n(T)-l(T)+s(T)-2}{3}-4\\[5pt]
		&=\tfrac{4(n(T')+3)-(l(T')+1)+(s(T')+1)-2}{3}-4=\tfrac{4n(T')-l(T')+s(T')-2}{3}.
	\end{aligned}
\end{equation*} 

\vspace{.15cm}

\noindent
From the previous subcases, we conclude that $\gamma_{t}^{oi}(\mathtt{S}(T'))=\frac{4n(T')-l(T')+s(T')-2}{3}$. Hence, by the induction hypothesis, it follows that $T'\in \mathcal{T}$. Since $v_4\in \mathcal{S}(T')\cup\mathcal{L}_w(T')$, the tree $T$ can be obtained from $T'$ by Operation $O_2$, and consequently, $T\in \mathcal{T}$.

\vspace{.15cm}

\noindent
\textbf{Case 2:}  $v_4\in\mathcal{SS}(T')\cup \mathcal{N}_{SS}(T')$. We first suppose that $v_4\in\mathcal{SS}(T')$. Let $v_k\in N_T(v_4)\cap \mathcal{S}(T)$ and let us consider that $N_T(v_k)\cap \mathcal{L}(T)=\{v_l\}$. Define $T'''=T-\{v_l\}$. Let $D'''$ be a $\gamma_{t}^{oi}(\mathtt{S}(T'''))$-set which satisfies the condition given in Lemma~\ref{lem-tree}, and without loss of generality, assume that $|D'''\cap N_{\mathtt{S}(T''')}(v_4)|$ is maximum. Since $v_2,v_4\in \mathcal{S}(T''')$ and $v_3\in \mathcal{SS}(T''')$, it follows that $v_2,v^{1,2},v_3,v_4,v^{4,k}\in D'''$ and $v_1,v^{2,3},v_k\notin D'''$. This implies that $(D'''\setminus\{v^{4,k}\})\cup \{v_k,v^{k,l}\}$ is a TOIDS of $\mathtt{S}(T)$. Hence, $\gamma_{t}^{oi}(\mathtt{S}(T))\leq |(D'''\setminus\{v^{4,k}\})\cup \{v_k,v^{k,l}\}|=\gamma_{t}^{oi}(\mathtt{S}(T'''))+1$. Moreover, we have that $n(T)=n(T''')+1$, $s(T)=s(T''')$, and $l(T)=l(T''')$. Hence,
\begin{equation*}
\begin{aligned}
\gamma_{t}^{oi}(\mathtt{S}(T'''))\geq\gamma_{t}^{oi}(\mathtt{S}(T))-1
&=\tfrac{4n(T)-l(T)+s(T)-2}{3}-1\\[5pt]
&=\tfrac{4(n(T''')+1)-l(T''')+s(T''')-2}{3}-1>\tfrac{4n(T''')-l(T''')+s(T''')-2}{3},
\end{aligned}
\end{equation*} 
which contradicts the upper bound given in Theorem~\ref{theo-tree}. Therefore, $v_4\in \mathcal{N}_{SS}(T')$. By Lemma~\ref{lem-equality-upper-bound} we have that $|N_T(v_i)\cap \mathcal{S}(T)|=1$ for every vertex $v_i\in N_T(v_4)\setminus \{v_5\}$. This implies that $T_{v_4}$ is isomorphic to $Q_r$, with $r=|N_T(v_4)|-1$. Let $A=V(T_{v_4})\setminus \{v_4\}$ and define $T^*=T-A$. Note that $v_4\in \mathcal{L}(T^*)$. If $v_4\in \mathcal{L}_s(T^*)$, then $v_5\in \mathcal{S}(T^*)$, which in turn implies that $v_5\in\mathcal{S}(T')$. Hence, $v_4\notin \mathcal{N}_{SS}(T')$, which is a contradiction. Therefore, $v_4\in \mathcal{L}_w(T^*)$. By Lemma~\ref{lem-tree-attaching}-(iv) we have that $\gamma_{t}^{oi}(\mathtt{S}(T))=\gamma_{t}^{oi}(\mathtt{S}(T^*))+4r$.
Moreover, $n(T)=n(T^*)+3r$, $s(T)=s(T^*)+r-1$, and $l(T)=l(T^*)+r-1$. By the hypothesis of the lemma and the previous equalities, we obtain that
\begin{equation*}
\begin{aligned}
\gamma_{t}^{oi}(\mathtt{S}(T^*))=\gamma_{t}^{oi}(\mathtt{S}(T))-4r
&=\tfrac{4n(T)-l(T)+s(T)-2}{3}-4r\\[5pt]	&=\tfrac{4(n(T^*)+3r)-(l(T^*)+r-1)+(s(T^*)+r-1)-2}{3}-4r\\[5pt]
&=\tfrac{4n(T^*)-l(T^*)+s(T^*)-2}{3}.
\end{aligned}
\end{equation*} 

\noindent
Hence, by the induction hypothesis, it follows that $T^*\in \mathcal{T}$. Since $v_4\in \mathcal{L}_w(T^*)$, the tree $T$ can be obtained from $T^*$ by Operation $O_3$, and consequently, $T\in \mathcal{T}$.
\end{proof}

\noindent
As an immediate consequence of Lemmas~\ref{lem-upper-bound-right} and \ref{lem-upper-bound-left} we have the desired characterization.

\begin{theorem}
Let $T$ be a nontrivial tree. Then $\gamma_{t}^{oi}(\mathtt{S}(T))=\frac{4n(T)-l(T)+s(T)-2}{3}$ if and only if  $T\in \mathcal{T}$.
\end{theorem}

%

\end{document}